\documentclass[12pt,twoside,reqno,psamsfonts]{amsart}

\usepackage[left=2.9cm,top=3cm,right=2.9cm]{geometry}              % See geometry.pdf to learn the layout options. There are lots.
\usepackage[colorlinks = true, citecolor = black, linkcolor = black, urlcolor = black, pdfstartview=FitH]{hyperref}  % links
\usepackage[pdftex]{graphicx}
\usepackage[utf8]{inputenc}
\usepackage{microtype, float}
\usepackage{amsmath, verbatim}
\usepackage{amssymb,mathtools}
\usepackage{epstopdf}
\usepackage{epsfig}
\usepackage{mathscinet}

\usepackage{marginfix}
\usepackage{marginnote}
\let\oldtocsection=\tocsection
\let\oldtocsubsection=\tocsubsection
\let\oldtocsubsubsection=\tocsubsubsection
\renewcommand{\tocsection}[2]{\hspace{0em}\textbf{\oldtocsection{#1}{#2}}}
\renewcommand{\tocsubsection}[2]{\hspace{1.8em}\oldtocsubsection{#1}{#2}}
\renewcommand{\tocsubsubsection}[2]{\hspace{3em}\oldtocsubsubsection{#1}{#2}}

\DeclareUnicodeCharacter{00A0}{ } % Tuomas: the strange 'space' character appearing on Mac OS X 
\numberwithin{equation}{section}

\theoremstyle{plain}
\newtheorem{thm}{Theorem}[section]
\newtheorem{theorem}[thm]{Theorem}

\newtheorem{lemma}[thm]{Lemma}

\theoremstyle{definition}

\theoremstyle{remark}
\newtheorem{remark}[thm]{Remark}

\renewcommand{\epsilon}{\varepsilon}
\renewcommand{\rho}{\varrho}
\renewcommand{\phi}{\varphi}

\renewcommand{\Im}{\mathrm{Im\,}}

\newcommand{\trdeg}{\text{tr.deg}}

\newcommand{\rexp}{\mathbb{R}_{\exp}}

   \def\XXint#1#2#3{{\setbox0=\hbox{$#1{#2#3}{\int}$}
        \vcenter{\hbox{$#2#3$}}\kern-.5\wd0}}

\title[A nondefinability result for Weierstrass $\wp$ functions]{A nondefinability result for expansions of the Field of Real Numbers by the restricted Weierstrass $\wp$ function}

\author{Raymond McCulloch}
\address{}
\email{raymond.mcculloch@postgrad.manchester.ac.uk}

 \subjclass[2010]{} %classifications

\begin{document}
	
	\begin{abstract}
Suppose that $\Omega$ is a complex lattice that is closed under complex conjugation and that $I$ is a small real interval, and that $D$ is a disc in $ \mathbb{C}$. Then the restriction $\wp|_D$ is definable in the structure $(\bar{\mathbb{R}},\wp|_I)$ if and only if the lattice $\Omega$ has complex multiplication. This characterises lattices with complex multiplication in terms of definability.

	\end{abstract}

\maketitle
	\section{Introduction}
Since the 1980's model theorists have studied expansions of the real ordered field $\bar{\mathbb{R}}=(\mathbb{R},<,+,\cdot,0,1)$ by various analytic functions. A major result in this area was due to Wilkie, who in \cite{wilkieexp} proved the model completeness of $\rexp=(\bar{\mathbb{R}},\exp)$ where $\exp$ is the real exponential function and this lead to much further activity. \par In particular in \cite{bianconi} Bianconi used Wilkie's ideas together with a functional transcendence result due to Ax to show that no arc of the sine function is definable in $\rexp$. If $\mathbb{C}$ is identified with $\mathbb{R}^2$, this result may be rephrased to say that no restriction  of the complex exponential function to an open disc in $\mathbb{C}$ is definable in $\rexp$. Bianconi went further in \cite{bianconiundef} and showed that if $f:D\rightarrow\mathbb{C}$ is holomorphic and definable in $\rexp$ then $f$ is algebraic. This result is used by Peterzil and Starchenko in \cite{newtonpetstar} to characterise all definable locally analytic subsets of $\mathbb{C}^n$ in the structure $\rexp$.
\par Formulated in this way, the question can be generalised to other functions. There certainly are transcendental functions $f$ such that there are discs $D\subseteq\mathbb{C}$ with $f|_D$ definable in the structure $(\bar{\mathbb{R}},f|_I)$ for some interval $I\subseteq\mathbb{R}$. Indeed it turns out that examples occur with functions not so different from $\exp$. Recall that to a lattice $\Omega\subseteq \mathbb{C}$ (i.e a discrete subgroup of rank 2) Weierstrass associated a meromorphic function, $$\wp(z)=\wp_\Omega(z)=\frac{1}{z^2}+\sum_{\substack{\omega\in\Omega^*}}\left(\frac{1}{(z-\omega)^2}-\frac{1}{\omega^2}\right),$$ where $\Omega^*=\Omega\setminus\{0 \}$. This meromorphic function $\wp$ has poles at exactly the points in $\Omega$ and is periodic with respect to $\Omega$. In Section \ref{background} we shall state the addition formula for $\wp$ and its differential equation. We can see that $\wp$ is similar to $\exp$ as they are both periodic and have an addition formula and a differential equation. They are also similar as they both give an exponential map of a commutative algebraic group. \par In the course of his investigations into the model theory of these elliptic functions, Macintyre observed in \cite{macwp} that if $\Omega=\mathbb{Z}+i\mathbb{Z}$ then in the structure $(\bar{\mathbb{R}},\wp|_{(1/8,3/8)})$ the restriction of $\wp$ to any disc $D$ on which $\wp$ is analytic is definable. (The interval (1/8,3/8) is chosen for convenience as it avoids both poles of $\wp$ and the zeros of $\wp^\prime$. Any such interval may be chosen.)
\par However the lattice $\Omega=\mathbb{Z}+i\mathbb{Z}$ is rather special. It can easily be seen that $\wp(iz)=-\wp(z)$. This formula is all that is required to prove the aforementioned observation of Macintyre in \cite{macwp}. In particular there are non-integer $\alpha \in \mathbb{C}$ such that $\alpha\Omega\subseteq\Omega$. (Any $\alpha \in i\mathbb{Z}$ for example). Such lattices are said to have \textit{complex multiplication}. The lattice $\Omega=\mathbb{Z}+i\mathbb{Z}$ is also closed under complex conjugation in other words $\bar{\Omega}=\Omega$. Lattices satisfying this latter property are known as \textit{real lattices}. It can be seen in Section 18 of \cite{duval} that if $\Omega$ is a real lattice then $\wp$ is real valued when restricted to an interval on either the real or imaginary axis. Hence such restrictions of $\wp$ are real analytic functions and it is natural to consider the model theory of these restrictions. In Section \ref{mac} it is shown that Macintyre's observation extends easily to all real lattices with complex multiplication. It can then be shown that this characterises those $\Omega$ for which $\wp|_I$ for some finite real interval $I$ defines $\wp|_D$ for some disc $D$. This is done in the following theorem.
\begin{theorem}\label{theorem}
	Let $\Omega$ be a real lattice and let $\wp_{\Omega}$ be its $\wp$ function. Let $I$ be some real interval not containing a pole. Then there's a non-empty disc $D \subseteq \mathbb{C}$ such that $\wp|_D$ is definable in the structure $(\bar{\mathbb{R}},\wp|_I)$ if and only if there's an $\alpha \in \mathbb{C}\setminus\mathbb{Z}$ such that $\alpha\Omega\subseteq\Omega$.
\end{theorem}
In Section \ref{theoremproof} we give the proof of Theorem \ref{theorem}. The theorem is proved by adapting the method of Bianconi in \cite{bianconi}. In \cite{jks} Jones, Kirby and Servi attempt to apply this method to answer questions about the local interdefinability of Weierstrass $\wp$ functions. However it turns out this method cannot be applied to their problem and they turn to the method of predimensions of Hrushovski in \cite{ehstrongmin}. In fact their method cannot be used here. The method of Bianconi involves using a theorem of Wilkie on smooth functions that are defined implicitly that was proved generally by Jones and Wilkie in \cite{jw}. Bianconi refers to these methods of Wilkie as the ``Desingularisation Theorem". We shall obtain this implicit  definition in Section \ref{impdefn} and also explain why this implicit definition may be used to prove Theorem \ref{theorem}.
\section{Background on the Weierstrass $\wp$ function}\label{background}
Here we give all the background that shall be needed on the Weierstrass $\wp$ function. All of this can be found in the books \cite{kc} or \cite{silverman}. The function $\wp$ is analytic except at its poles which are at exactly the points in the lattice $\Omega$. Also $\wp$ is doubly periodic and $\Omega$ is its period lattice. Clearly $\wp$ depends on $\Omega$. As mentioned in the introduction to this paper $\wp$ has an addition formula, which is now stated. Let $z,w$ be complex numbers such that $z-w \notin \Omega$. Then, \begin{equation}
\label{wpadd}\wp(z+w)=\frac{1}{4}\left(\frac{\wp^{\prime}(z)-\wp^{\prime}(w)}{\wp(z)-\wp(w)}\right)^2-\wp(z)-\wp(w).
\end{equation}This gives rise to the duplication formula,  \begin{equation}\label{wpdup}
\wp(2z)=\frac{1}{4}\left(\frac{\wp^{\prime\prime}(z)}{\wp^{\prime}(z)}\right)^2 -2\wp(z).
\end{equation} Another important property of $\wp$ is that it satisfies a differential equation namely \begin{equation}\label{wpdiff}
(\wp^{\prime}(z))^2=4\wp^3(z)-g_2\wp(z)-g_3,
\end{equation} where the complex numbers $g_2$ and $g_3$ are called the \textit{invariants of $\wp$}. From this differential equation it is clear that there is an algebraic relation between $\wp$ and $\wp^{\prime}$. Differentiating both sides of this equation gives us that \begin{equation}\label{wp2diff}
\wp^{\prime\prime}(z)=6\wp^2(z)-\frac{g_2}{2}.
\end{equation} \par An \textit{elliptic function} $f$ is a meromorphic function on $\mathbb{C}$ which has two complex periods $\omega$ and $ \omega^{\prime}$ such that $\Im (\omega^{\prime}/\omega)>0$. These periods generate a lattice of periods $\Omega$. If we denote the field of elliptic functions for a fixed period lattice $\Omega$ by $L$ then it is known that this field $L$ is in fact $\mathbb{C}(\wp,\wp^\prime)$. This can be seen in Theorem 3.2 in Chapter 6 of \cite{silverman}.
\begin{remark}\label{reallattices}
	From Section 2 of \cite{duval} it can be seen that real lattices may be split into two cases which are known as the \textit{rectangular} and \textit{rhombic} lattices. A lattice $\Omega$ is called a \textit{rectangular lattice} if we can choose generators $\omega_1$ and $\omega_2$ of $\Omega$ such that $\omega_1$ is real and $\omega_2$ is purely imaginary. The lattice $\Omega$ is called a \textit{rhombic lattice} if we can choose generators of $\Omega$, namely $\omega_1$ and $\omega_2$ such that $\overline{\omega_1}=\omega_2$.\\ In Section \ref{theoremproof} we give the proof of Theorem \ref{theorem} in the rectangular lattice case, the proof of the rhombic case is similar.
\end{remark}
\par Finally we conclude this background section by stating a version of Ax's theorem for the Weierstrass $\wp$ function. This is a theorem of Ax in \cite{axadg1} and Brownawell and Kubota in \cite{bkpaper}.

\begin{theorem}\label{BK}
	Let $\Omega$ be a complex lattice which does not have complex multiplication. Let $z_1,\dots,z_n$ be power series with complex coefficients and no constant term that are linearly independent over $\mathbb{Q}$. Then we have that $$\trdeg_{\mathbb{C}}\mathbb{C}[z_1,\dots,z_n,\wp(z_1),\dots,\wp(z_n)]\ge n+1.$$
\end{theorem}

\section{Macintyre's Result}\label{mac}
Recall that if a complex lattice $\Omega$ has complex multiplication then there's a non-integer $\alpha$ such that $\alpha \Omega\subseteq\Omega$.
\begin{lemma}\label{maclemma}
	Let $\Omega$ be a real lattice that is closed under complex conjugation and let $\wp_\Omega$ be its $\wp$ function. Let $I$ be some real interval not containing a pole. Then the restriction of $\wp$ to any complex disc not meeting any poles is definable in the structure $(\bar{\mathbb{R}},\wp|_I)$.
\end{lemma}
\begin{proof}
	We follow the proof of Macintyre in \cite{macwp} for the case $\Omega=\mathbb{Z}+i\mathbb{Z}$. This just checks the proof works in the general case. Using the extra symmetry of the lattice $\Omega$, due to complex multiplication, we now show that $\wp$ restricted to $\alpha I$, where $\alpha\in\mathbb{C}\setminus\mathbb{Z}$ is such that $\alpha\Omega\subseteq\Omega$, is definable in the structure $(\bar{\mathbb{R}},\wp|_I)$. Let $z \in I$ and $f(z)=\wp(\alpha z).$ Then for any $\omega \in \Omega$, we can see that, $$f(z+\omega)=\wp(\alpha z+\alpha\omega)=\wp(\alpha z)=f(z)$$ as $\alpha \Omega\subseteq\Omega$. Therefore $f$ is a meromorphic function that is doubly periodic with respect to the lattice $\Omega$. So $f$ is an elliptic function for the lattice $\Omega$. Hence $f$ may be written as a rational function $R$ in $\wp(z)$ and $\wp^\prime(z)$. Similarly the function $g(z)=\wp^\prime(\alpha z)$ may be written as a rational function $S$ in $\wp(z)$ and $\wp^\prime(z)$. \par Therefore both the functions $\wp$ and $\wp^\prime$ restricted to $\alpha I$ are definable in the structure $(\bar{\mathbb{R}},\wp|_I)$. Now consider some disc $D$ contained in $I\times\alpha I$. For $z\in D$ it is clear that we may write $z=x+\alpha y$ for $x,y\in I$. We can assume that $x-\alpha y \notin \Omega.$  Then by the addition formula, $$\wp(z)=\wp(x+\alpha y)=\frac{1}{4}\left(\frac{\wp^\prime(x)-S(\wp(y),\wp^\prime(y))}{\wp(x)-R(\wp(y),\wp^\prime(y))} \right)^2-\wp(x)-R(\wp(y),\wp^\prime(y)).$$ As every function in this expression is definable we have that the function $\wp|_D$ is definable in the structure $(\bar{\mathbb{R}},\wp|_I)$. Using the addition formula gives us that $\wp$ restricted to any disc in $\mathbb{C}$ is definable in $(\bar{\mathbb{R}},\wp|_I)$ as required. 
\end{proof}
\section{Getting an implicit definition}\label{impdefn}
In this section we shall describe the implicit definition that shall be needed in the proof of Theorem \ref{theorem}. This implicit definition is due to Wilkie in \cite{wilkieexp} and is referred to by Bianconi as the "Desingularisation Theorem" and was proved in a more general form by Jones and Wilkie in \cite{jw}. We let $\tilde{\mathbb{R}}=(\bar{\mathbb{R}},\mathcal{F})$ be an expansion of $\bar{\mathbb{R}}$ by a set $\mathcal{F}$ of total analytic functions in one variable, closed under differentiation that is o-minimal with a model complete theory. Desingularisation can be used to show that if the function $f:U\rightarrow \mathbb{R}$, for some $U\subseteq\mathbb{R}$, is a definable function of $\tilde{\mathbb{R}}$ then $f$ may be defined piecewise by a system of equations whose matrix of partial derivatives is non-singular. More precisely there's some interval $I^\prime \subseteq U$, some integer $n \ge 1,$ and there are certain functions $F_1,\dots,F_n:\mathbb{R}^{n+1}\rightarrow \mathbb{R}$ (see below) and functions $f_2,\dots,f_n:I^\prime\rightarrow \mathbb{R}$ such that for all $t\in I^\prime$,  $$\begin{array}{ccc}
F_1(t,f(t),f_2(t),\dots,f_n(t))=0\\\vdots\\ F_{n}(t,f(t),f_2(t),\dots,f_n(t))=0
\end{array}$$ and $$\det \left(\frac{\partial F_i}{\partial x_j}\right)_{\substack{i=1,\dots,n\\j=2,\dots,n+1}}(t,f(t),f_2(t),\dots,f_n(t))\ne 0.$$
\par Here the functions $F_1,\dots,F_n$ are polynomials in $x_1,\dots,x_{n+1}$ and $g_i(x_j)$ for $i=1,\dots,l$ and $j=1,\dots,n+1$ where $g_1,\dots,g_l\in \mathcal{F}$.  \par To apply this to our setting we can see from \eqref{wp2diff} that there is a polynomial expression for $\wp^{\prime\prime}$ in terms of $\wp$ and its invariants. Using this expression and the differential equation \eqref{wpdiff} for $\wp$ we may find similar expressions in $\wp$ and $\wp^{\prime}$ for every derivative of $\wp.$ Hence we have that the ring of terms with parameters from $\mathbb{R}$ of the structure $(\bar{\mathbb{R}},\wp|_I)$ is closed under differentiation. The structure $(\bar{\mathbb{R}},\wp|_I)$ is model complete by a theorem of Gabrielov in \cite{gab}. Bianconi also has model completeness results involving the $\wp$ function in \cite{bianconimc} however these seem difficult to apply here as they are for the complex functions rather than their restrictions to a real interval. 
\par The main obstacle to using this implicit definition in the proof of Theorem \ref{theorem} is that the function $\wp|_I$ is not a total function on $\mathbb{R}.$ In order to use this method these functions are made to be total functions on $\mathbb{R}$. We do this by composing $\wp$ with a bijection from $\mathbb{R}$ to the interval $I$. This bijection is defined for a general interval $I=(a,b)$ here. In the proof of Theorem \ref{theorem} we give an explicit interval but do not explicitly define the bijection. Firstly define $A:(a,b)\rightarrow\mathbb{R}$ by $$A(t)=\frac{t-\frac{b+a}{2}}{\left(\frac{b-a}{2}\right)^2-\left(t-\frac{b+a}{2}\right)^2},$$ which is a bijection. Differentiating gives, \begin{align*}
A^\prime (t)&= \frac{\left(\frac{b-a}{2}\right)^2+\left(t-\frac{b+a}{2}\right)^2}{\left( \left(\frac{b-a}{2}\right)^2-\left(t-\frac{b+a}{2}\right)^2\right)^2},
\end{align*}which clearly doesn't vanish. The inverse, $B=A^{-1}$ is also differentiable and $$B^\prime(t)=\frac{\left(\left(\frac{b-a}{2}\right)^2-\left(B(t)-\frac{b+a}{2}\right)^2\right)^2}{\left(\frac{b-a}{2}\right)^2+\left(B(t)-\frac{b+a}{2}\right)^2}$$ also doesn't vanish. 
Finally we define $$B_1(t)=\frac{1}{\left(\frac{b-a}{2}\right)^2+\left(B(t)-\frac{b+a}{2}\right)^2}.$$Therefore the structure $(\bar{\mathbb{R}},\wp\circ B,B,B_1)$ is an expansion of $\bar{\mathbb{R}}$ by total analytic functions on $\mathbb{R}$. The structures $(\bar{\mathbb{R}},\wp|_I)$ and $(\bar{\mathbb{R}},\wp\circ B,B,B_1)$ are equivalent in the sense that they have the same definable sets. Hence the structure $(\bar{\mathbb{R}},\wp\circ B,B,B_1)$ is also model complete and has a ring of terms with parameters from $\mathbb{R}$ that is closed under differentiation. Therefore in the proof of Theorem \ref{theorem} it is sufficient to pass from the structure $(\bar{\mathbb{R}},\wp|_I)$ to the auxiliary structure $(\bar{\mathbb{R}},\wp\circ B,B,B_1)$ and prove the theorem in this structure. In the next section it is explained how we may use this implicit definition in the structure $(\bar{\mathbb{R}},\wp\circ B,B,B_1)$ in order to prove Theorem \ref{theorem}.
\section{Proof of Theorem \ref{theorem}}\label{theoremproof}
In this section we give the proof of Theorem \ref{theorem}. The argument uses a method of Bianconi in \cite{bianconi} as well as \cite{bianconiundef}, which involves finding upper and lower bounds on the transcendence degree of some finite extension of $\mathbb{C}$ that are contradictory. The lower bound is found using Theorem \ref{BK} and the upper bound requires a similar argument to that used to prove Claim 4 in the proof of Theorem 4 in \cite{bianconiundef}.
\begin{proof}[Proof of Theorem \ref{theorem}.]
	By Remark \ref{reallattices} we know that real lattices may be divided into two cases, the rectangular and rhombic lattices. Firstly we note that one direction of Theorem \ref{theorem} is proved in Lemma \ref{maclemma}. Here we prove the other direction of Theorem \ref{theorem} for the rectangular lattice case as both cases are fairly similar. The main difference is the choice of generators of $\Omega$ that may be made and the intervals that can be chosen accordingly. The full details of both cases will appear in my PhD thesis. 
	\par Let $\Omega$ be a rectangular real lattice. From Section 19 of \cite{duval} we can see that we can choose generators for the lattice $\Omega$, namely $\omega_1$ and $\omega_2$ such that $\omega_1$ is real and $\omega_2$ is purely imaginary. This leads to two cases, when $|\omega_1|\le |\omega_2|$ and when $|\omega_2|\le |\omega_1|$. We shall want to find a real interval $I$ not containing any lattice points such that $iI$ also doesn't contain any lattice points. In the first case the interval $(\omega_1/8,3\omega_1/8)$ is chosen and in the second case the interval $(\omega_2/8i,3\omega_2/8i)$ is chosen. It suffices to prove the theorem in either case and so we assume we are in the first of these cases and that $I=(\omega_1/8,3\omega_1/8)$. Due to the discussion in the previous section we know that it suffices to prove the theorem in the auxiliary structure $(\bar{\mathbb{R}},\wp\circ B,B,B_1)$ and we therefore now pass to this structure.

	\par Now we assume for a contradiction that $\Omega$ doesn't have complex multiplication and that there's a non-empty disc $D\subseteq\mathbb{C}$ such that the restriction $\wp|_D$ is definable in the structure $(\bar{\mathbb{R}},\wp\circ B,B,B_1)$. By translating and scaling using the addition formula \eqref{wpadd} we may take $D$ to contain the interval $iI=(i\omega_1/8,3i\omega_1/8)$. As $|\omega_1|\le |\omega_2|$ the interval $iI$ doesn't contain any lattice points. Hence the function $f_1:I\rightarrow \mathbb{R}$ defined by $f_1(t)=\wp(iB(t))$ is definable in the structure $(\bar{\mathbb{R}},\wp\circ B,B,B_1)$. By the discussion in Section \ref{impdefn} there's an open interval $I^\prime\subseteq(\omega_1/8,3\omega_1/8)$, some integer $n \ge 1,$ certain functions $F_1,\dots,F_n:\mathbb{R}^{n+1}\rightarrow\mathbb{R}$ and functions $f_2,\dots,f_n:I^\prime\rightarrow\mathbb{R}$ such that for all $t\in I^\prime$, $$\begin{array}{ccc}
	F_1(t,f_1,\dots,f_n)=0\\\vdots\\F_n(t,f_1,\dots,f_n)=0
	\end{array}$$ and $$\det\left(\frac{\partial F_i}{\partial x_j}\right)_{\substack{i=1,\dots,n\\j=2,\dots,n+1}}(t,f_1,\dots,f_n)\ne 0.$$ The functions $F_1,\dots,F_n$ are polynomials in \begin{align*}
	&x_1,\dots,x_{n+1},\wp(B(x_1)),\dots,\wp(B(x_{n+1})),\\&B(x_1),\dots,B(x_{n+1}),B_1(x_1),\dots,B_1(x_{n+1}).
	\end{align*} However as $B$ and $B_1$ are algebraic and $\wp$ and $\wp^\prime$ are algebraically dependent we may take the functions $F_1,\dots,F_n$ to be algebraic in $x_1,\dots,x_{n+1}$ and $\wp(B(x_1)),\dots,\wp(B(x_{n+1}))$ after potentially moving to a smaller interval.\par Take $n$ to be minimal such that there exists an interval $I^\prime$ and functions $F_1,\dots,F_n$ in $x_1,\dots,x_{n+1}$ and $\wp(B(x_1)),\dots,\wp(B(x_{n+1}))$ such that the above system of equations and non-singularity condition holds. As we are restricted to a small interval it is possible to shrink this interval to avoid the zeros of $\wp^\prime$ if necessary. In fact the interval may be shrunk repeatedly if needed.\par The desired contradiction will be obtained by considering upper and lower bounds on the transcendence degree of a certain finite extension of $\mathbb{C}$. These bounds will be incompatible.\par Firstly we find an lower bound on transcendence degree. In order to find this lower bound we want to apply Theorem \ref{BK}. To apply this theorem it must first be shown that $B(t)-B(0),B(f_1(t))-B(f_1(0)),\dots,B(f_n(t))-B(f_n(0))$ are linearly independent over $\mathbb{Q}$. As the lattice $\Omega$ doesn't have complex multiplication this linear independence is only required over $\mathbb{Q}.$
	\par This linear independence is shown by showing that linear dependence contradicts the minimality of $n$. This linear independence argument does not depend on whether the lattice $\Omega$ is rectangular or rhombic. If $B(t)-B(0),B(f_1(t))-B(f_1(0)),\dots,B(f_n(t))-B(f_n(0))$ are assumed to be linearly dependent over $\mathbb{Q}$ then it can be shown that $B(f_n(t))-B(f_n(0))$ can be written as a $\mathbb{Q}$-linear combination of $B(t)-B(0),B(f_1(t))-B(f_1(0)),\dots,B(f_{n-1}(t))-B(f_{n-1}(0))$. For any natural number $n$ it is possible to obtain a formula $\wp(nz)$ in terms of $\wp(z)$ and $\wp^\prime(z)$ using the addition and duplication formulas \eqref{wpadd} and $\eqref{wpdup}$. Similarly we can write $\wp(z)$ as an algebraic expression in $\wp(z/n)$ and $\wp^\prime(z/n)$.  Hence there's an algebraic expression for $\wp(z/n)$ in terms of $\wp(z)$ and $\wp^\prime(z)$. Therefore the functions $F_1,\dots,F_n$ may be rewritten as functions in $t,f_1,\dots,f_{n-1}$. This gives a new system of equations in $n$ variables. If this new system has a matrix of partial derivatives that's non-zero at $(t,f_1,\dots,f_{n-1})$ then there's a contradiction to the minimality of $n$. Hence it is assumed that all the minors of the matrix of partial derivatives for this system have determinant equal to zero and from this assumption a contradiction to the non-singularity of the original system is obtained. \par Now we observe that adding $i(B(t)-B(0))$ to the list $B(t)-B(0),B(f_1(t))-B(f_1(0)),\dots,B(f_n(t))-B(f_n(0))$ doesn't destroy this linear independence. For if it did then we could write for rational $a_0,\dots,a_{n}$ not all zero, $$i(B(t)-B(0))=a_0(B(t)-B(0))+a_1(B(f_1(t))-B(f_1(0)))+\dots+a_n(B(f_n(t))-B(f_n(0))).$$ The left hand side of this expression is non-real whereas the right hand side is real for all $t$ in the interval $I^\prime$, a contradiction. Now applying Theorem \ref{BK} to $iB(t),B(t),B(f_1(t)),\dots,B(f_n(t))$ gives that \begin{align*}
	\trdeg_{\mathbb{C}}\mathbb{C}[&iB(t),B(t),B(f_1(t)),\dots,B(f_n(t)),\wp(iB(t)),\\&\qquad\wp(B(t)),\wp(B(f_1(t))),\dots,\wp(B(f_n(t)))]\ge n+3.
	\end{align*}
	Now we obtain a upper bound on the transcendence degree of the same finite extension of $\mathbb{C}$ that is smaller than this lower bound. Recall that on the interval $I^\prime$ the function $f_1(t)=\wp(iB(t))$ is defined by a system of equations $F_1,\dots,F_n$ in the variables $x_1,\dots,x_{n+1}$, whose Jacobian is non-singular. For each $i=1,\dots,n$ the function $F_i$ can be written as $$ F_i(t,f_1,\dots,f_n)=P_i(t,f_1,\dots,f_n,\wp(B(t)),\wp(B(f_1(t))),\dots,\wp(B(f_n(t)))),$$ where $P_1,\dots,P_n$ are algebraic functions in $y_1,\dots,y_{2n+2}$. This is a system of $n$ equations in $2n+2$ variables. We check that the matrix $$\left(\frac{\partial P_i}{\partial y_j}\right)_{\substack{i=1,\dots,n\\j=2,\dots,2n+2}}(t,f_1,\dots,f_n,\wp(B(t)),\wp(B(f_1(t))),\dots,\wp(B(f_n(t))))$$ is of maximal rank $n$. This is done by a similar argument to that of Claim 4 in the proof of Theorem 4 in \cite{bianconiundef}. Differentiating $F_i$ with respect to $x_j$ gives $$\frac{\partial F_i}{\partial x_j}=\frac{\partial P_i}{\partial y_j}+B^{\prime}(x_j)\wp^{\prime}(B(x_j))\frac{\partial P_i}{\partial y_{j+n+1}}.$$ Therefore the matrix $(\partial F_i/\partial x_j)$ is given by multiplying the matrix $(\partial P_i/\partial y_j)$ by a $(2n+1)\times n$ matrix $M$.
	
	Here $M$ is the matrix whose first $n\times n$ block is the identity matrix followed by a row of zeros and finally an $n \times n$ diagonal matrix diag$(B^{\prime}(x_2)\wp^{\prime}(B(x_2)),\dots,B^{\prime}(x_{n+1})\wp^{\prime}(B(x_{n+1}))).$ In other words $M$ is the matrix $$M=\begin{pmatrix}
	1&\dots&0\\\vdots&\ddots&\vdots\\0&\dots&1\\0&\dots &0\\B^\prime(x_2)\wp^\prime(B(x_2))
	&\dots&0\\\vdots&\ddots&\vdots\\0&\dots&B^\prime(x_{n+1})\wp^\prime(B(x_{n+1}))
	\end{pmatrix}.$$ By the non-singularity of the system $F_1,\dots,F_n$ it's clear that the rows of the matrix $(\partial F_i/\partial x_j)$ are linearly independent over $\mathbb{R}$ and so the rows of the matrix $(\partial P_i/\partial y_j)$ are also linearly independent over $\mathbb{R}$. So this matrix is of maximal rank $n$ as required. Therefore $$ \trdeg_{\mathbb{C}}\mathbb{C}[t,f_1(t),\dots,f_n(t),\wp(B(t)),\wp(B(f_1)),\dots,\wp(B(f_n))]\le n+2.$$ However this is an upper bound on the transcendence degree on a slightly different extension of $\mathbb{C}$. But $iB(t)$ and $B(t)$ are algebraically dependent as are $B(\wp(iB(t)))$ and $\wp(iB(t))$ and $f_i$ and $B(f_i)$ for all $i=1,\dots,n$ as $B$ is algebraic. Therefore we have that \begin{align*}
	\trdeg_{\mathbb{C}}\mathbb{C}[iB(t),\wp(iB(t)),B(t)&,\wp(B(t)),B(f_1(t)),\dots,B(f_n(t)),\\&\qquad\wp(B(f_1(t))),\dots,\wp(B(f_n(t)))]\le n+2.
	\end{align*}So we have found upper and lower bounds on the transcendence degree of some finite extension of $\mathbb{C}$ that are incompatible. Hence we have a contradiction and the theorem is proved.
\end{proof}
%\begin{remark}
%	It turns out that the assumption that the lattice $\Omega$ is real can be removed. The proof of Lemma \ref{maclemma} is the same however the proof of the converse requires some alteration. This is due to the upper bound becoming too large and therefore the lower bound must be increased. This can be done by adding another $\wp$ function to the structure and applying the full version of Brownawell and Kubota in \cite{bk}. The details of this will appear in my PhD thesis.
%\end{remark}

\bibliographystyle{abbrv}

\bibliography{A_nondefinability_result_for_expansions_of_the_ordered_real_field_by_the_weierstrass_p_function} %put your bib file name as long as it is in the same file directory
\end{document}